\documentclass[11pt]{article}

\usepackage[T1]{fontenc}
\usepackage[utf8]{inputenc}
\usepackage{lmodern}
\usepackage[english]{babel}
\usepackage{amsmath,amssymb,amsthm}
\numberwithin{equation}{section}
\usepackage{graphicx}
\usepackage{xcolor}
\usepackage{caption}
\usepackage{booktabs}
\usepackage[margin=2.5cm]{geometry}
\usepackage[colorlinks=true,linkcolor=black,citecolor=black,urlcolor=magenta]{hyperref}
\usepackage[all]{hypcap}
\hypersetup{pdftitle={Admissible qx+1 Sequences, Semiconvergents, and Rational Catalan Numbers},pdfauthor={Mike Winkler}}

\theoremstyle{plain}
\newtheorem{theorem}{Theorem}
\newtheorem{corollary}[theorem]{Corollary}
\newtheorem{lemma}[theorem]{Lemma}
\newtheorem{proposition}[theorem]{Proposition}
\newtheorem{conjecture}[theorem]{Conjecture}
\theoremstyle{definition}
\newtheorem{definition}[theorem]{Definition}
\theoremstyle{remark}
\newtheorem{remark}[theorem]{Remark}

\newcommand{\seqnum}[1]{\href{https://oeis.org/#1}{\textup{\underline{#1}}}}
\newcommand{\Aq}{\mathcal A}
\newcommand{\Lq}{\mathcal L}
\newcommand{\Uq}{\mathcal U}
\newcommand{\Cq}{\mathcal C}
\newcommand{\Dq}{\mathcal D}
\newcommand{\Hq}{\mathcal H}
\newcommand{\Wq}{\mathcal K}
\newcommand{\Cat}{\operatorname{Cat}}

\title{Admissible $qx+1$ Sequences, Semiconvergents, and Rational Catalan Numbers}
\author{Mike Winkler\\
	Faculty of Mathematics\\
	Ruhr University Bochum, Germany\\
	\href{mailto:mike.winkler@ruhr-uni-bochum.de}{\texttt{mike.winkler@ruhr-uni-bochum.de}}}

\date{October 22, 2026}

\begin{document}
	
	\maketitle
	
	\begin{abstract}
		Let $q\geq3$ be odd and put $m_r=\lfloor r\log_2q\rfloor$. We count parity sequences with exactly $r$ entries equal to $1$ for which the coefficient in the $qx+1$ iteration first drops below $1$ at step $m_r+1$. Their number $a_q(r)$ also counts the corresponding residue classes modulo $2^{m_r+1}$. By classifying increasing integer sequences according to the first index at which they exceed a prescribed bound, we obtain a recurrence which, for $q=3$, establishes the five- or six-term block pattern conjectured in 2015, allowing a shorter final block. We establish the sharp inequalities
		\[
		\frac1r\binom{m_r-1}{r-1}\leq a_q(r)\leq\frac1r\binom{m_r}{r-1}
		\]
		Equality in the lower (respectively upper) bound occurs exactly when $\{r\log_2q\}$ is smaller (respectively larger) than all preceding fractional parts $\{j\log_2q\}$, $1\leq j<r$. These indices are characterized by denominators of convergents and semiconvergents from below and above; at these indices $r\geq2$, the values $a_q(r)$ are rational Catalan numbers. We also characterize the dependence on $q$ and prove, for each fixed odd $q\geq3$, that $a_q(r)=\Theta_q(r^{-3/2}B_q^r)$, where $B_q=(\log_2q)^{\log_2q}/(\log_2q-1)^{\log_2q-1}$. For $q=3$, this proves Hikawa's conjectured polynomial factor for the number of finite binary words of given weight for which the length of every nonempty prefix is smaller than $\log_2 3$ times the number of its ones.
	\end{abstract}

	\section{Introduction} \label{sec:introduction}
	
	The $3x+1$ problem asks whether every positive integer eventually reaches $1$ under iteration of the map $T(n)=n/2$ for even $n$ and $T(n)=(3n+1)/2$ for odd $n$; see Lagarias \cite{Lagarias85} for a survey. A basic tool in its study, introduced by Terras \cite{Terras76} and Everett \cite{Everett77}, is the sequence of parities of $n,T(n),T^2(n),\ldots$. For an odd integer $q\geq3$ we consider the analogous map
	\begin{equation}
		T_q(n)=
		\begin{cases}
			n/2, & n \text{ even},\\
			(qn+1)/2, & n \text{ odd},
		\end{cases} \label{eq:1-1}
	\end{equation}
	on the positive integers, so that $T_3=T$. Each application of $T_q$ multiplies $n$ by $q/2$ or $1/2$ and adds a nonnegative constant. Hence, if exactly $j$ of the integers $n,T_q(n),\ldots,T_q^{k-1}(n)$ are odd, then $T_q^k(n)=(q^j/2^k)\,n+\rho$, where $\rho\geq0$ depends only on the parities of these integers (see \eqref{eq:2-1} below). We denote by $\tau_q(n)$ the \emph{coefficient stopping time} of $n$, namely the least $k\geq1$ for which $q^j<2^k$, with $\tau_q(n)=\infty$ if no such $k$ exists. For $q=3$, Terras~\cite[Definition~1.5]{Terras76} introduced this quantity as the $\tau$-stopping time; we use the term \emph{coefficient stopping time}, as in Lagarias~\cite{Lagarias85}, for all odd $q$. For $q=3$, the \emph{ordinary stopping time} $\sigma(n)$ is the least $k\geq1$ such that $T^k(n)<n$, when such a $k$ exists. Terras conjectured that $\tau_3(n)=\sigma(n)$ for every $n\geq2$~\cite[Conjecture~2.9]{Terras76}.
	
	Whether the coefficient stopping time of $n$ equals $k$ depends only on the parities of $n,\ldots,T_q^{k-1}(n)$. Section~\ref{sec:parity} recalls that, for each $k$, these parity sequences of length $k$ are in bijection with the residue classes modulo $2^k$ (Lemma~\ref{lem:terras}). Counting residue classes with a given coefficient stopping time therefore reduces to counting finite $0$-$1$ sequences $v=(v_1,\ldots,v_L)$ such that
	\[
	q^{v_1+\cdots+v_\ell}>2^\ell \quad(1\leq\ell<L), \qquad q^{v_1+\cdots+v_L}<2^L.
	\]
	We call such a sequence \emph{$q$-admissible}, and its \emph{order} is the number of indices $i$ with $v_i=1$. Let $a_q(r)$ be the number of $q$-admissible sequences of order $r$. If we replace each entry $1$ by the factor $q/2$ and each entry $0$ by $1/2$, then the conditions state that every proper initial product exceeds $1$ and the full product is less than $1$. In this form, and for $q=3$, these sequences were called admissible by Wagon \cite{Wagon85}, who used them to study the average stopping time of the $3x+1$ map. For $q=3,5,7$, the numbers $a_q(r)$, $r\geq1$, are the OEIS sequences \seqnum{A100982}, \seqnum{A174795}, and \seqnum{A174796} \cite{OEIS}. Every $q$-admissible sequence of order $r\geq1$ has length $\lfloor r\log_2q\rfloor+1$ (Lemma~\ref{lem:structure}), so $a_q(r)$ is the number of residue classes modulo $2^{\lfloor r\log_2q\rfloor+1}$ whose elements have coefficient stopping time $\lfloor r\log_2q\rfloor+1$ (Proposition~\ref{prop:density}).
	
	Throughout, $q\geq3$ is odd, $\alpha=\log_2q$, and
	\[
	m_i=\lfloor i\alpha\rfloor \qquad(i\geq0).
	\]
	The number $\alpha$ is irrational because $q$ is not a power of $2$. The main results are as follows.
	
	\emph{A recurrence.} The OEIS entry \seqnum{A100982} contains a recurrence for $a_3(n)$ that the author conjectured in \cite{WinStopping15}. Its summands involve an integer parameter $b$ that is constant on blocks of five or six consecutive summation indices, apart from a possibly shorter final block (Conjecture~\ref{conj:2015}). The earlier paper describes the count in terms of ordinary stopping time, whereas the present residue-class count uses coefficient stopping time. For every fixed $k$, the sets of integers with ordinary and coefficient stopping time $k$ have finite symmetric difference; this follows from Terras~\cite[Proposition~1.8, Theorem~1.9, and the proof of Theorem~1.11]{Terras76}, as summarized by Lagarias~\cite[Theorem~C(b)]{Lagarias85}. The PARI/GP algorithm recorded in the OEIS implements \eqref{eq:3-5} with a particular choice of $b$ through order $53$. Corollary~\ref{cor:blockwise-conjecture} gives an explicit parameter for every order, with the same block lengths. The proof uses a general identity for increasing integer sequences bounded by a nondecreasing sequence, obtained by classifying the sequences that violate the bound according to the first index at which they do so (Proposition~\ref{prop:first-violation-range}).
	
	\emph{Bounds and their equality cases.} We prove (Theorem~\ref{thm:extremal-orders})
	\[
	\frac1r\binom{m_r-1}{r-1}\leq a_q(r)\leq\frac1r\binom{m_r}{r-1}.
	\]
	The two binomial coefficients count the compositions of $m_r$ and of $m_r+1$ into $r$ positive parts. The factor $1/r$ comes from the cycle lemma of Dvoretzky and Motzkin \cite{DvoretzkyMotzkin47}, which concerns the $r$ cyclic shifts of such a composition (Lemma~\ref{lem:rotation}). Earlier bounds of Nakamigawa and Tokushige and of Hikawa are discussed in Remark~\ref{rem:earlier-bounds}. The upper bound and the characterization of both equality cases require a separate argument. Equality holds in the lower bound if and only if the fractional part $\{r\alpha\}$ is smaller than $\{j\alpha\}$ for every $1\leq j<r$, and in the upper bound if and only if it is larger than $\{j\alpha\}$ for every $1\leq j<r$. By a theorem of Han\v{c}l and Turek \cite{HanclTurek}, the integers $r$ with the first property are exactly the denominators of the convergents and semiconvergents of $\alpha$ that lie below $\alpha$, and those with the second property are the denominators of the convergents and semiconvergents that lie above $\alpha$ (Theorem~\ref{thm:cf}; the terms are recalled in Section~\ref{sec:cf}).
	
	\emph{Rational Catalan numbers.} For these orders $r\geq2$, the $q$-admissible sequences of order $r$ are in bijection with the lattice paths from $(0,0)$ to $(r,s)$ that stay weakly below the line through these two points, where $s=m_r-r$ or $s=m_r+1-r$, and $\gcd(r,s)=1$. By a classical result of Bizley \cite{Bizley54}, the number of such paths is $\binom{r+s}{r}/(r+s)$, a rational Catalan number (Theorem~\ref{thm:dyck}). For instance, $a_3(7)=\binom{11}{7}/11=30$.
	
	\emph{Dependence on $q$ and growth.} The number $a_q(r)$ is nondecreasing in $q$, and we characterize equality (Proposition~\ref{prop:monotonicity}). Hikawa previously obtained the exponential growth rate for $q=3$~\cite[Theorem~6.3]{HikawaJuly26}. We obtain a two-sided estimate with the factor $r^{-3/2}$ for every odd $q\geq3$ (Corollary~\ref{cor:theta}), thereby proving Hikawa's conjectured refinement for $q=3$~\cite[Conjecture~9.1]{HikawaSeptember26}.
	
	Section~\ref{sec:parity} gives the connection with the map $T_q$ and the structure of admissible sequences. Section~\ref{sec:recurrence} contains the recurrence, Sections~\ref{sec:extremal} and~\ref{sec:cf} the bounds and their equality cases, and Sections~\ref{sec:catalan}, \ref{sec:comparison}, and~\ref{sec:growth} the remaining results.

	\section{Parity sequences and admissible sequences} \label{sec:parity}
	
	For a positive integer $n$ and $i\geq1$, let $v_i(n)\in\{0,1\}$ be the residue of $T_q^{i-1}(n)$ modulo $2$. The sequence $(v_1(n),\ldots,v_k(n))$ is the \emph{parity sequence of length $k$} of $n$. Put $j_k(n)=v_1(n)+\cdots+v_k(n)$. Since $T_q(x)=x/2$ or $T_q(x)=(q/2)x+1/2$, induction on $k$ gives
	\begin{equation}
		T_q^k(n)=\frac{q^{j_k(n)}}{2^k}\,n+\rho_k(n), \label{eq:2-1}
	\end{equation}
	where $\rho_0(n)=0$ and $\rho_{k+1}(n)$ equals $\rho_k(n)/2$ if $v_{k+1}(n)=0$ and $(q\rho_k(n)+1)/2$ if $v_{k+1}(n)=1$. In particular, $\rho_k(n)\geq0$ depends only on the parity sequence of length $k$ of $n$. The coefficient stopping time of $n$ is
	\begin{equation}
		\tau_q(n)=\min\left\{k\geq1:\ q^{j_k(n)}<2^k\right\}, \label{eq:2-2}
	\end{equation}
	with $\tau_q(n)=\infty$ if no such $k$ exists. Note that $q^{j}\neq2^k$ for $k\geq1$, since $q$ is odd.
	
	The following lemma is due to Terras \cite{Terras76} and Everett \cite{Everett77} for $q=3$; see also Lagarias~\cite[Theorem~B]{Lagarias85}. The proof applies to every odd $q$, and we include it for completeness.
	
	\begin{lemma} \label{lem:terras}
		Let $k\geq1$. Two positive integers that are congruent modulo $2^k$ have the same parity sequence of length $k$. Moreover, each of the $2^k$ sequences in $\{0,1\}^k$ is the parity sequence of length $k$ of the elements of exactly one residue class modulo $2^k$.
	\end{lemma}
	
	\begin{proof}
		We prove the first assertion by induction on $k$; the case $k=1$ is clear. Let $k\geq2$ and $n\equiv n'\pmod{2^k}$. Then $n$ and $n'$ have the same parity, and $T_q(n)-T_q(n')$ equals $(n-n')/2$ or $q(n-n')/2$. In both cases $T_q(n)\equiv T_q(n')\pmod{2^{k-1}}$, and the induction hypothesis applied to $T_q(n)$ and $T_q(n')$ gives the claim.
		
		Consequently, the parity sequence of length $k$ defines a map from the $2^k$ residue classes modulo $2^k$ to $\{0,1\}^k$. Both sets have $2^k$ elements, so it suffices to show that the map is injective. If $n$ and $n'$ have the same parity sequence of length $k$, with $j$ entries equal to $1$, then \eqref{eq:2-1} gives
		\[
		T_q^k(n)-T_q^k(n')=\frac{q^{j}}{2^k}\,(n-n').
		\]
		The left-hand side is an integer and $q$ is odd, so $2^k$ divides $n-n'$.
	\end{proof}
	
	By \eqref{eq:2-2}, $\tau_q(n)=L$ holds if and only if the parity sequence $(v_1,\ldots,v_L)$ of length $L$ of $n$ satisfies the conditions of the following definition.
	
	\begin{definition} \label{def:admissible}
		A sequence $v=(v_1,\ldots,v_L)\in\{0,1\}^L$ with $L\geq1$ is \emph{$q$-admissible} if
		\begin{equation}
			q^{v_1+\cdots+v_\ell}>2^\ell \quad(1\leq\ell<L), \qquad q^{v_1+\cdots+v_L}<2^L. \label{eq:2-3}
		\end{equation}
		The \emph{order} of $v$ is the number of indices $i$ with $v_i=1$. Let $\Aq_q(r)$ denote the set of $q$-admissible sequences of order $r$, and put $a_q(r)=|\Aq_q(r)|$.
	\end{definition}
	
	As explained in the introduction, Definition~\ref{def:admissible} agrees for $q=3$ with Wagon's definition \cite{Wagon85} after the substitution $1\mapsto q/2$, $0\mapsto1/2$. For $q=3$, the $3$-admissible sequences of positive order are the ``terminal'' sequences of Terras~\cite[Definition~1.12]{Terras76}. His term ``admissible'' has a weaker meaning; see also the historical explanation by P.~Cochin in \seqnum{A100982}.\footnote{In Terras' Definition~1.12 the prefix range is printed as $1\leq i\leq k-2$. His Theorem~1.14 and its proof require the range $1\leq i\leq k-1$.} For $q=3,5,7$ we have, with offset $1$,
	\[
	a_3(r)=\seqnum{A100982}(r),\qquad a_5(r)=\seqnum{A174795}(r),\qquad a_7(r)=\seqnum{A174796}(r).
	\]
	As of October 2026, the formula fields of \seqnum{A174795} and \seqnum{A174796} state the condition on the initial products only for $1<\ell<L$, but the examples given there show that the condition for $\ell=1$ is intended. With this reading, all three entries agree with Definition~\ref{def:admissible}.
	
	The only $q$-admissible sequence of order $0$ is $(0)$. The next lemma describes the sequences of positive order. Recall that $m_i=\lfloor i\alpha\rfloor$ with $\alpha=\log_2q$. Since $\alpha>1$, we have $m_{i+1}-m_i\geq1$ for all $i\geq0$.
	
	For $q=3$, the length assertion in the following lemma already follows from the proof of Terras~\cite[Theorem~2.1]{Terras76}; it was later recorded by T.~D.~Noe in the OEIS entry \seqnum{A100982} and appears in Hikawa~\cite[Lemma~3.1]{HikawaJuly26}. An earlier description of the block compositions in \eqref{eq:2-6} and the corresponding residue classes for the unaccelerated $3n+1$ map is given by Garner~\cite[Section~4]{Garner81}. We retain a proof for general odd $q$.
	
	\begin{lemma} \label{lem:structure}
		Let $r\geq1$ and let $v\in\{0,1\}^L$ contain the entry $1$ exactly $r$ times, at the positions $p_1<p_2<\cdots<p_r$. Then $v$ is $q$-admissible if and only if
		\begin{equation}
			L=m_r+1,\qquad p_1=1,\qquad p_{i+1}-1\leq m_i \quad(1\leq i<r). \label{eq:2-4}
		\end{equation}
		Consequently, the map $v\mapsto(s_1,\ldots,s_{r-1})$ with $s_i=p_{i+1}-1$ is a bijection from $\Aq_q(r)$ onto the set of integer sequences satisfying
		\begin{equation}
			1\leq s_1<s_2<\cdots<s_{r-1}, \qquad s_i\leq m_i\quad(1\leq i<r). \label{eq:2-5}
		\end{equation}
		For $r=1$ this set consists of the empty sequence only.
	\end{lemma}
	
	\begin{proof}
		For $1\leq\ell\leq L$ put $E_\ell=(v_1+\cdots+v_\ell)\alpha-\ell$, so that \eqref{eq:2-3} states $E_\ell>0$ for $\ell<L$ and $E_L<0$. The value $E_\ell$ increases by $\alpha-1>0$ when $v_\ell=1$ and decreases by $1$ when $v_\ell=0$.
		
		Assume that $v$ is $q$-admissible. Since $r\geq1$, we have $L\geq2$, because $L=1$ would force $v=(0)$. Then $E_1>0$ gives $v_1=1$, that is, $p_1=1$. If $v_L=1$, then $E_L>E_{L-1}>0$, which is impossible; hence $v_L=0$, and $E_{L-1}=E_L+1$. Thus $r\alpha-(L-1)>0>r\alpha-L$, which gives $L-1<r\alpha<L$, that is, $L=m_r+1$. Finally, for $1\leq i<r$, the index $\ell=p_{i+1}-1$ satisfies $\ell<L$ and $v_1+\cdots+v_\ell=i$. Hence $i\alpha-(p_{i+1}-1)>0$, and since $i\alpha$ is not an integer, $p_{i+1}-1\leq m_i$.
		
		Conversely, assume \eqref{eq:2-4}. On each interval $p_i\leq\ell<p_{i+1}$ (with $p_{r+1}=L+1$) the value $E_\ell$ decreases, so on the range $1\leq\ell<L$ the minima of $E_\ell$ can occur only at $\ell=p_{i+1}-1$ with $1\leq i<r$ and at $\ell=L-1$. At $\ell=p_{i+1}-1$ we have $E_\ell=i\alpha-(p_{i+1}-1)\geq i\alpha-m_i>0$. Since $p_r=s_{r-1}+1\leq m_{r-1}+1\leq m_r=L-1$ for $r\geq2$, and $p_1=1\leq m_1$, the last $1$ occurs before position $L$, so $E_{L-1}=r\alpha-m_r>0$ and $E_L=r\alpha-m_r-1<0$. Hence $v$ is $q$-admissible.
		
		The inequalities $p_1<p_2<\cdots<p_r$ translate into $1\leq s_1<\cdots<s_{r-1}$, and \eqref{eq:2-4} determines $v$ from $(s_1,\ldots,s_{r-1})$. This proves the bijection.
	\end{proof}
	
	It is often convenient to describe an admissible sequence by the lengths of its runs. Write $v\in\Aq_q(r)$ as
	\[
	v=1\,0^{b_1-1}\,1\,0^{b_2-1}\cdots1\,0^{b_r-1},
	\]
	where $0^c$ denotes $c$ consecutive zeros. Then $(b_1,\ldots,b_r)$ is a \emph{composition} of $L=m_r+1$ into $r$ parts, that is, a sequence of $r$ positive integers with sum $m_r+1$. Its partial sums $B_i=b_1+\cdots+b_i$ satisfy $B_i=p_{i+1}-1=s_i$ for $i<r$. By Lemma~\ref{lem:structure}, $v\mapsto(b_1,\ldots,b_r)$ is a bijection from $\Aq_q(r)$ onto the set of compositions of $m_r+1$ into $r$ parts with
	\begin{equation}
		B_i\leq m_i\qquad(1\leq i<r). \label{eq:2-6}
	\end{equation}
	We call these compositions \emph{admissible}.
	
	\begin{proposition} \label{prop:density}
		Every finite value of $\tau_q$ is either $1$ or of the form $m_r+1$ with $r\geq1$. The positive integers $n$ with $\tau_q(n)=1$ are the even integers. For $r\geq1$, the positive integers $n$ with $\tau_q(n)=m_r+1$ form exactly $a_q(r)$ residue classes modulo $2^{m_r+1}$. In particular, this set has natural density $a_q(r)/2^{m_r+1}$.
	\end{proposition}
	
	\begin{proof}
		We have $\tau_q(n)=L$ if and only if the parity sequence of length $L$ of $n$ is $q$-admissible. By Lemma~\ref{lem:structure}, the order $r$ of this sequence determines $L$: $L=1$ for $r=0$ and $L=m_r+1$ for $r\geq1$. Since $r\mapsto m_r$ is strictly increasing, conversely $L$ determines $r$. Lemma~\ref{lem:terras} now shows that the residue classes modulo $2^{m_r+1}$ with $\tau_q=m_r+1$ correspond to the elements of $\Aq_q(r)$.
	\end{proof}
	
	For $q=3$, the possible finite values of $\tau_3$ were characterized by Terras~\cite[Theorem~2.1]{Terras76}, and his Proposition~1.6 establishes constancy on residue classes modulo $2^k$. The residue-class interpretation is also given by Lagarias~\cite[Theorem~C(a)]{Lagarias85}; see Garner~\cite[Section~4]{Garner81}. Garner uses the unaccelerated map, for which $r$ odd steps contribute $r$ additional iterations: his coefficient stopping time is $\kappa=(m_r+1)+r$, while the residue-class modulus is still $2^{m_r+1}$.
	
	\begin{remark} \label{rem:total-density}
		By Proposition~\ref{prop:density}, the set of positive integers with $\tau_q(n)\leq k$ is a finite union of residue classes, and
		\[
		F_q=\frac12+\sum_{r\geq1}\frac{a_q(r)}{2^{m_r+1}}
		\]
		is the limit of its density as $k\to\infty$. By Lemma~\ref{lem:terras}, the parity sequence of length $k$ of a residue class chosen uniformly at random modulo $2^k$ is uniformly distributed on $\{0,1\}^k$. Therefore $F_q$ is the probability that the random walk $E_\ell=(v_1+\cdots+v_\ell)\alpha-\ell$, with independent uniformly distributed $v_i\in\{0,1\}$, becomes negative for some $\ell\geq1$. Its steps have mean $(\alpha-2)/2$.
		
		For $q=3$ the mean is negative, $E_\ell\to-\infty$ almost surely by the strong law of large numbers, and $F_3=1$. Thus almost all positive integers, in the sense of density, have finite coefficient stopping time for $q=3$; the corresponding density-one result for ordinary stopping time is due to Terras and was also proved independently by Everett; see Lagarias~\cite[Theorem~A]{Lagarias85}.
		
		For $q\geq5$ the mean is positive, so $\inf_\ell E_\ell>-\infty$ almost surely, and there is $c>0$ with $\Pr(\inf_\ell E_\ell\geq-c)>0$. Choose $N$ with $N(\alpha-1)>c$. With probability $2^{-N}$ the first $N$ entries equal $1$, and, independently, the walk started after step $N$ stays above $-c$ relative to its starting value with positive probability. On the intersection of these two events $E_\ell>0$ for all $\ell$. Hence $F_q<1$ for $q\geq5$. A floating-point computation of the probabilities up to $k=12000$ gives $F_5\approx0.8240$ and $F_7\approx0.6992$; these two values are numerical and are not used below.
	\end{remark}

	\section{A recurrence for \texorpdfstring{$a_q(r)$}{a\_q(r)}} \label{sec:recurrence}
	
	By Lemma~\ref{lem:structure}, $a_q(r)$ counts increasing sequences bounded by the sequence $(m_i)$. The identity in this section holds for an arbitrary nondecreasing bound. Let $\mu=(\mu_1,\mu_2,\ldots)$ be a nondecreasing sequence of nonnegative integers. For $r\geq2$ put
	\begin{equation}
		N_r(\mu)= \#\left\{(s_1,\ldots,s_{r-1})\in\mathbb Z^{r-1}: 1\leq s_1<\cdots<s_{r-1},\ s_i\leq\mu_i\ (1\leq i<r)\right\}, \label{eq:3-1}
	\end{equation}
	and $N_1(\mu)=1$. For real $\alpha>1$ we write $c_r(\alpha)=N_r\bigl((\lfloor i\alpha\rfloor)_{i\geq1}\bigr)$. Lemma~\ref{lem:structure} states that
	\begin{equation}
		a_q(r)=c_r(\log_2q)\qquad(r\geq1). \label{eq:3-2}
	\end{equation}
	Terras~\cite[Theorem~1.14]{Terras76} gave an earlier recurrence involving binomial coefficients, indexed by word length and the number of zeros. In Terras' terminology, $n(a,k)$ counts his admissible words of length $k$ with $a$ zeros, using the corrected prefix range noted after Definition~\ref{def:admissible}; Lemma~\ref{lem:structure} yields the specialization $a_3(r)=n(m_r+1-r,m_r+1)$. This specialization is not stated explicitly by Terras. His length-indexed recurrence differs from the order-indexed first-violation decomposition below.
	
	Classical enumeration methods for paths under general boundaries include the determinant formula of Kreweras~\cite{Kreweras65}, presented in Krattenthaler~\cite[Theorem~10.7.1]{Krattenthaler15}. For the bounds $\mu_i=\lfloor i\alpha\rfloor$, this formula counts the same sequences as $N_r(\mu)$. The following first-violation decomposition applies directly to an arbitrary nondecreasing integer bound; see also Mohanty~\cite{Mohanty79} for lattice path enumeration. The case $q=3$ of \eqref{eq:3-3} with $\mu_i=\lfloor i\log_23\rfloor$ also occurs in a Lean~4 formalization by R.~H.~G. van Tol (private communication). We use the convention $\binom{u}{v}=0$ for $0\leq u<v$.
	
	\begin{proposition} \label{prop:first-violation-range}
		Let $\mu$ be a nondecreasing sequence of nonnegative integers, let $r\geq2$, and let $M\geq\mu_{r-1}$ be an integer. Then
		\begin{equation}
			\binom{M}{r-1}=N_r(\mu)+\sum_{i=1}^{r-1}N_i(\mu)\binom{M-\mu_i}{r-i}. \label{eq:3-3}
		\end{equation}
	\end{proposition}
	
	\begin{proof}
		The left-hand side counts the sequences $1\leq s_1<\cdots<s_{r-1}\leq M$. Since $M\geq\mu_{r-1}\geq\mu_i$ for $i<r$, the sequences with $s_i\leq\mu_i$ for all $i<r$ are exactly those counted by $N_r(\mu)$. Every other sequence has a smallest index $i$ with $s_i>\mu_i$. For a fixed $i$, the initial segment $(s_1,\ldots,s_{i-1})$ is one of the $N_i(\mu)$ sequences counted in \eqref{eq:3-1} with $r$ replaced by $i$. If $i\geq2$, then $s_{i-1}\leq\mu_{i-1}\leq\mu_i$. Hence the remaining entries $s_i<\cdots<s_{r-1}$ can be any $r-i$ distinct elements of $\{\mu_i+1,\ldots,M\}$, listed in increasing order, and conversely every such choice, appended to an initial segment counted by $N_i(\mu)$, gives a sequence whose first violation occurs at $i$. This gives $N_i(\mu)\binom{M-\mu_i}{r-i}$ sequences for each $i$, and \eqref{eq:3-3} follows.
	\end{proof}
	
	For the rest of this section let $q=3$, and write $a(n)=a_3(n)$ and $m_j=\lfloor j\log_23\rfloor$, with $n$ for the order as in the OEIS.
	
	\begin{corollary} \label{cor:A100982-recurrence}
		For every $n\geq2$,
		\begin{equation}
			a(n)=\binom{m_{n-1}-1}{n-2} -\sum_{i=2}^{n-1}\binom{m_{n-1}-m_i}{n-i}a(i). \label{eq:3-4}
		\end{equation}
	\end{corollary}
	
	\begin{proof}
		Apply Proposition~\ref{prop:first-violation-range} with $\mu_i=m_i$, $r=n$, and $M=m_{n-1}$, and use $a(i)=N_i(\mu)$ from \eqref{eq:3-2}. Since $m_1=1$ and $N_1(\mu)=1$, the term $i=1$ equals $\binom{m_{n-1}-1}{n-1}$, and
		\[
		\binom{m_{n-1}}{n-1}-\binom{m_{n-1}-1}{n-1} =\binom{m_{n-1}-1}{n-2}
		\]
		by Pascal's rule. Solving for $a(n)$ gives \eqref{eq:3-4}.
	\end{proof}
	
	The recurrence conjectured in \cite{WinStopping15}, as stated in \seqnum{A100982}, takes the following form when the terminal truncation implicit in its finite-sum example is made explicit.
	
	\begin{conjecture}[\cite{WinStopping15}] \label{conj:2015}
		For $n\geq4$,
		\begin{equation}
			a(n)=\frac{(m+n-2)!}{m!\,(n-2)!}-\sum_{i=2}^{n-1}\binom{\lfloor(3(n-i)+b)/2\rfloor}{n-i}a(i), \qquad m=m_{n-1}-(n-1), \label{eq:3-5}
		\end{equation}
		where $b$ is an integer depending on $n$ and $i$ that, for fixed $n$, takes the same value on consecutive intervals of five or six summation indices $i$, except that the final interval may be shorter.
	\end{conjecture}
	
	Since $m+n-2=m_{n-1}-1$, the first term in \eqref{eq:3-5} equals $\binom{m_{n-1}-1}{n-2}$, as in \eqref{eq:3-4}. For $n=13$, the example in \seqnum{A100982} uses $b=0$ for $2\leq i\leq6$, $b=-1$ for $7\leq i\leq11$, and $b=-2$ for $i=12$. The OEIS entry also gives a PARI/GP algorithm, explicitly limited to $n\leq53$. The exception for the final block is necessary already for $4\leq n\leq6$, when fewer than five summation indices occur. The next corollary gives a parameter valid for all $n$.
	
	\begin{corollary} \label{cor:blockwise-conjecture}
		Let $\theta=2\log_2 3-3=\log_2(9/8)$. For $n\geq3$ and $2\leq i\leq n-1$, put
		\begin{equation}
			b_{n,i}=2m_{n-1}-3n+1-\lfloor i\theta\rfloor+\mathbf{1}_{\{i=6\}}, \label{eq:3-6}
		\end{equation}
		where $\mathbf{1}_{\{i=6\}}$ is $1$ if $i=6$ and $0$ otherwise. Then
		\begin{equation}
			\left\lfloor\frac{3(n-i)+b_{n,i}}2\right\rfloor=m_{n-1}-m_i. \label{eq:3-7}
		\end{equation}
		Consequently \eqref{eq:3-5} holds for every $n\geq3$ with $b=b_{n,i}$. For fixed $n$, the sequence $(b_{n,i})_{2\leq i\leq n-1}$ is nonincreasing, consecutive values differ by $0$ or $1$, and every complete constant block starting at $i=2$ has length five or six. Only the final block, truncated at $i=n-1$, may be shorter.
	\end{corollary}
	
	\begin{proof}
		Put $\alpha=\log_2 3$, so that $\theta=2\alpha-3$. If $i\neq6$, we have
		\[
		\lfloor2i\alpha\rfloor=3i+\lfloor i\theta\rfloor=2m_i+e_i,\qquad e_i\in\{0,1\},
		\]
		and therefore
		\[
		3(n-i)+b_{n,i}=2(m_{n-1}-m_i)+1-e_i.
		\]
		This proves \eqref{eq:3-7} for $i\neq6$. For $i=6$, we have $m_6=9$ and $\lfloor6\theta\rfloor=1$, whence
		\[
		3(n-6)+b_{n,6}=2m_{n-1}-17
		\]
		and \eqref{eq:3-7} follows in this case as well. Without the correction at $i=6$, \eqref{eq:3-7} would still hold, but for $n\geq7$ the first constant block would be $\{2,\ldots,5\}$, of length four. Substitution in \eqref{eq:3-4} proves \eqref{eq:3-5}.
		
		Since $(9/8)^5<2<(9/8)^6$, we have $5<1/\theta<6$. Direct integer comparisons give $\lfloor5\theta\rfloor=0$, $\lfloor6\theta\rfloor=\lfloor11\theta\rfloor=1$, and $\lfloor12\theta\rfloor=2$. Hence $b_{n,i}$ is constant on $i=2,\ldots,6$ and on $i=7,\ldots,11$, except where the sum ends earlier; each of these blocks has length five when complete. For $i\geq12$, the constant blocks are the sets of integers in $[u/\theta,(u+1)/\theta)$, $u\geq2$. Such an interval has length strictly between five and six and therefore contains five or six integers. Because $0<\theta<1$, successive block values decrease by exactly one. This proves the assertion about block lengths and monotonicity.
	\end{proof}
	
	\begin{remark} \label{rem:old-conjecture}
		The PARI/GP algorithm published with \seqnum{A100982} on January 3, 2018, uses, for its stated range $7\leq n\leq53$, the parameters
		\[
		b^{\mathrm{old}}_{n,i}=2\left\lfloor\frac{19(n-1)}{12}\right\rfloor-3n+1-\left\lfloor\frac{i}{6}\right\rfloor+\mathbf{1}_{\{i=6\}}.
		\]
		Here $19/12$ is a convergent of $\log_2 3$, and $2(19/12)-3=1/6$. The integer parts in this formula agree with those in \eqref{eq:3-6} throughout that range. If the old algorithm is extended beyond its stated limit, the first discrepancy occurs at $n=54$: $\lfloor53\log_2 3\rfloor=84$, whereas $\lfloor53\cdot19/12\rfloor=83$, and $\lfloor53\theta\rfloor=9$, whereas $\lfloor53/6\rfloor=8$. The parameter in \eqref{eq:3-6} thus extends the historical pattern to all orders. In particular, for $n=13$ it gives $b_{13,i}=0$ for $2\leq i\leq6$, $b_{13,i}=-1$ for $7\leq i\leq11$, and $b_{13,12}=-2$, exactly as in the OEIS example. The parameter is not uniquely determined: for each $n$ and $i$, exactly two consecutive integer values of $b$ satisfy the stronger termwise identity \eqref{eq:3-7}, and even the partition into blocks need not be unique (for example, when $n=18$). The recurrence \eqref{eq:3-4} also differs from the alternating-sum recurrence of V.~M.~Zarubin recorded in \seqnum{A100982}.
	\end{remark}

	\section{The cycle lemma and bounds for \texorpdfstring{$a_q(r)$}{a\_q(r)}} \label{sec:extremal}
	
	The binomial coefficients $\binom{m_r-1}{r-1}$ and $\binom{m_r}{r-1}$ count compositions of $m_r$ and $m_r+1$ into $r$ parts. To compare these with the admissible compositions of Section~\ref{sec:parity}, we use the cycle lemma, which concerns the cyclic shifts of a composition.
	
	Let $x=(x_1,\ldots,x_r)$ be a composition of a positive integer $T$ into $r$ parts, and let $X_i=x_1+\cdots+x_i$. For $0\leq k<r$, the \emph{rotation} of $x$ by $k$ is the composition $(x_{k+1},\ldots,x_r,x_1,\ldots,x_k)$. The \emph{rotation class} of $x$ is the set of its distinct rotations; its number of elements divides $r$. We say that $x$ lies \emph{below the diagonal} if
	\[
	X_i\leq\frac{iT}{r}\qquad(1\leq i<r),
	\]
	and \emph{strictly below the diagonal} if these inequalities are strict. Geometrically, the points $(i,X_i)$ then lie on or below (respectively strictly below, for $0<i<r$) the line from $(0,0)$ to $(r,T)$.
	
	The following lemma is a form of the cycle lemma of Dvoretzky and Motzkin \cite{DvoretzkyMotzkin47}; see Dershowitz and Zaks \cite{DershowitzZaks90} for its history and several applications, and Graham, Knuth, and Patashnik \cite[Section~7.5]{GKP94} for the closely related lemma of Raney \cite{Raney60}. Related cycle-lemma bounds are compared in Remark~\ref{rem:earlier-bounds}.
	
	\begin{lemma}[Cycle lemma] \label{lem:rotation}
		Let $x$ be a composition of $T$ into $r$ parts.
		\begin{itemize}
			\item[(a)] The rotation class of $x$ contains a composition that lies below the diagonal.
			\item[(b)] The rotation class of $x$ contains at most one composition that lies strictly below the diagonal. If it contains one, then it has exactly $r$ elements.
			\item[(c)] If $\gcd(r,T)=1$, then the rotation class of $x$ has exactly $r$ elements and contains exactly one composition that lies strictly below the diagonal.
		\end{itemize}
	\end{lemma}
	
	\begin{proof}
		Put $S_i=rX_i-iT$ for $0\leq i\leq r$. Then $S_0=S_r=0$, and we extend $S$ to all integers $i\geq0$ by $S_{i+r}=S_i$. The rotation of $x$ by $k$ has partial sums $X_{k+i}-X_k$ (with $X_{i+r}=X_i+T$), so the corresponding values $S$ are $S_{k+i}-S_k$. Hence the rotation by $k$ lies below the diagonal if and only if $S_{k+i}\leq S_k$ for $0<i<r$, and strictly below if and only if $S_{k+i}<S_k$ for $0<i<r$.
		
		(a) Choose $k\in\{0,\ldots,r-1\}$ with $S_k$ maximal.
		
		(b) If the rotation class of $x$ has fewer than $r$ elements, then $x$ has a period $p<r$ dividing $r$, so $X_p=pT/r$, and the same holds for every rotation of $x$; thus no rotation lies strictly below the diagonal. If the class has $r$ elements, then the rotations by $0,\ldots,r-1$ are distinct, and a rotation by $k$ lies strictly below the diagonal if and only if $S_k$ is strictly larger than all other values $S_0,\ldots,S_{r-1}$. At most one $k$ has this property.
		
		(c) If $\gcd(r,T)=1$, the numbers $S_i\equiv-iT\pmod r$, $0\leq i<r$, are pairwise distinct, so the maximum of $S_0,\ldots,S_{r-1}$ is attained exactly once. Moreover, a period $p<r$ would give $r\mid pT$ and hence $r\mid p$, which is impossible. Now (b) gives the claim.
	\end{proof}
	
	The equality cases of our bounds are described by the following records. For irrational $\alpha>1$ and $i\geq1$ put
	\[
	\delta_i=\{i\alpha\}=i\alpha-\lfloor i\alpha\rfloor\in(0,1).
	\]
	These numbers are pairwise distinct. We call $r\geq2$ a \emph{lower record index} of $\alpha$ if $\delta_r<\delta_j$ for all $1\leq j<r$, and an \emph{upper record index} if $\delta_r>\delta_j$ for all $1\leq j<r$. For $\alpha=\log_2q$ we have $q^r/2^{m_r}=2^{\delta_r}$, so $r$ is a lower record index if and only if $q^r/2^{m_r}<q^j/2^{m_j}$ for all $1\leq j<r$, and an upper record index if and only if $2^{m_r+1}/q^r<2^{m_j+1}/q^j$ for all $1\leq j<r$.
	
	For $q=3$, Terras~\cite[Table~C]{Terras76} tabulated certain upper record indices in his stopping-time notation. Garner~\cite[Table~1]{Garner81} tabulated the lower and upper record indices of $\log_2 3$ up to $31867$, encoded as the indices at which his two auxiliary maxima increase. These earlier numerical results are not used in the characterization below.
	
	The next lemma translates the record properties into conditions on integer parts. Conditions of this type occur in the theory of Christoffel words; see Berstel and S\'e\'ebold \cite{BerstelSeebold02}.
	
	\begin{lemma} \label{lem:record-floor}
		Let $\alpha>1$ be irrational, $r\geq2$, $m_j=\lfloor j\alpha\rfloor$, $m=m_r$, and $t=m+1$. Then $r$ is a lower record index of $\alpha$ if and only if
		\begin{equation}
			\gcd(r,m)=1 \qquad\text{and}\qquad m_j=\left\lfloor\frac{mj}{r}\right\rfloor \quad(1\leq j<r). \label{eq:4-1}
		\end{equation}
		Similarly, $r$ is an upper record index of $\alpha$ if and only if
		\begin{equation}
			\gcd(r,t)=1 \qquad\text{and}\qquad m_j=\left\lfloor\frac{tj}{r}\right\rfloor \quad(1\leq j<r). \label{eq:4-2}
		\end{equation}
	\end{lemma}
	
	\begin{proof}
		Let $r$ be a lower record index. If $d>1$ divides both $r$ and $m$, then $j=r/d$ satisfies $j\alpha=m/d+\delta_r/d$, so $\delta_j=\delta_r/d<\delta_r$, a contradiction. Hence $\gcd(r,m)=1$. Since $m/r<\alpha$, we have $\lfloor mj/r\rfloor\leq m_j$. If this inequality were strict for some $j<r$, then $m_j>mj/r$ and
		\[
		\delta_j=j\alpha-m_j <j\alpha-\frac{mj}{r} =\frac{j}{r}\delta_r<\delta_r,
		\]
		again a contradiction. Thus \eqref{eq:4-1} holds.
		
		Conversely, assume \eqref{eq:4-1}, and suppose that $\delta_j<\delta_r$ for some $1\leq j<r$. From $j\alpha+(r-j)\alpha=r\alpha$ we get $\delta_j+\delta_{r-j}\in\{\delta_r,\delta_r+1\}$, and the second value is impossible because $\delta_j<\delta_r$ and $\delta_{r-j}<1$. Hence $\delta_j+\delta_{r-j}=\delta_r$. On the other hand, \eqref{eq:4-1} gives $m_j\leq mj/r$ and $m_{r-j}\leq m(r-j)/r$, that is,
		\[
		\delta_j\geq\frac{j}{r}\delta_r, \qquad \delta_{r-j}\geq\frac{r-j}{r}\delta_r.
		\]
		Since the two sides add up to $\delta_r$, both inequalities are equalities, so $m_j=mj/r$ and $r\mid mj$. As $\gcd(r,m)=1$, this forces $r\mid j$, which is impossible. Hence $r$ is a lower record index.
		
		Now let $r$ be an upper record index. If $d>1$ divides both $r$ and $t$, then $j=r/d$ satisfies $j\alpha=t/d-(1-\delta_r)/d$, so $1-\delta_j=(1-\delta_r)/d<1-\delta_r$, a contradiction. Hence $\gcd(r,t)=1$. Since $t/r>\alpha$, we have $m_j\leq\lfloor tj/r\rfloor$. If this inequality were strict for some $j<r$, then $m_j+1\leq tj/r$ and
		\[
		1-\delta_j=m_j+1-j\alpha \leq\frac{tj}{r}-j\alpha =\frac{j}{r}(1-\delta_r)<1-\delta_r,
		\]
		a contradiction. Thus \eqref{eq:4-2} holds.
		
		Conversely, assume \eqref{eq:4-2}, and suppose that $\delta_j>\delta_r$ for some $1\leq j<r$. Then $\delta_j+\delta_{r-j}>\delta_r$, hence $\delta_j+\delta_{r-j}=\delta_r+1$, that is,
		\[
		(1-\delta_j)+(1-\delta_{r-j})=1-\delta_r.
		\]
		Since $\lfloor y\rfloor>y-1$ for every real $y$, \eqref{eq:4-2} gives $m_j+1>tj/r$ and $m_{r-j}+1>t(r-j)/r$, that is,
		\[
		1-\delta_j>\frac{j}{r}(1-\delta_r), \qquad 1-\delta_{r-j}>\frac{r-j}{r}(1-\delta_r).
		\]
		Adding these inequalities contradicts the previous identity. Hence $r$ is an upper record index.
	\end{proof}
	
	We return to $\alpha=\log_2q$. Since $\alpha\geq\log_23>3/2$, we have
	\begin{equation}
		m_r\geq r+1\qquad(r\geq2). \label{eq:4-3}
	\end{equation}
	Let $\Lq_q$ (respectively $\Uq_q$) be the set of integers $r\geq1$ for which equality holds in the lower (respectively upper) bound of the following theorem.
	
	\begin{theorem} \label{thm:extremal-orders}
		For every odd $q\geq3$ and every $r\geq1$,
		\begin{equation}
			\frac1r\binom{m_r-1}{r-1} \leq a_q(r)\leq \frac1r\binom{m_r}{r-1}. \label{eq:4-4}
		\end{equation}
		Both bounds are attained for $r=1$. For $r\geq2$, equality holds in the lower bound if and only if $r$ is a lower record index of $\log_2q$, and in the upper bound if and only if $r$ is an upper record index of $\log_2q$. In particular, $\Lq_q\cap\Uq_q=\{1\}$.
	\end{theorem}
	
	\begin{proof}
		For $r=1$ all three terms in \eqref{eq:4-4} equal $1$. Let $r\geq2$, $m=m_r$, and $t=m+1$. Let $\Wq$ be the set of admissible compositions, that is, compositions of $t$ into $r$ parts satisfying \eqref{eq:2-6}; by Section~\ref{sec:parity}, $|\Wq|=a_q(r)$.
		
		\emph{The lower bound.} Let $\Cq$ be the set of compositions of $m$ into $r$ parts, so $|\Cq|=\binom{m-1}{r-1}$, and let $\Hq\subseteq\Cq$ consist of the compositions $h$ with partial sums $H_i\leq m_i$ for $1\leq i<r$. Every $b\in\Wq$ has last part $b_r=t-B_{r-1}\geq m_r+1-m_{r-1}\geq2$. Hence subtracting $1$ from the last part is a bijection from $\Wq$ onto $\Hq$, and $a_q(r)=|\Hq|$.
		
		If $h\in\Cq$ lies below the diagonal, then $H_i\leq im/r<i\alpha$, so $H_i\leq m_i$ and $h\in\Hq$. By Lemma~\ref{lem:rotation}(a), every rotation class in $\Cq$ meets $\Hq$. Each class has at most $r$ elements, so there are at least $|\Cq|/r$ classes, and $a_q(r)=|\Hq|\geq|\Cq|/r$. Equality holds if and only if every class has exactly $r$ elements and meets $\Hq$ in exactly one element.
		
		If $d=\gcd(r,m)>1$, then repeating a composition of $m/d$ into $r/d$ parts $d$ times gives a composition whose class has fewer than $r$ elements (such a composition exists because $m\geq r$), so the inequality is strict. Let $\gcd(r,m)=1$. Then every class has $r$ elements and contains exactly one composition strictly below the diagonal, by Lemma~\ref{lem:rotation}(c). Since $mi/r$ is not an integer for $0<i<r$, a composition $h$ lies strictly below the diagonal if and only if $H_i\leq\lfloor mi/r\rfloor$ for $i<r$.
		\begin{itemize}
			\item If $m_i=\lfloor mi/r\rfloor$ for all $i<r$, then $\Hq$ is precisely the set of compositions strictly below the diagonal, and $|\Hq|=|\Cq|/r$.
			\item Otherwise $m_i>\lfloor mi/r\rfloor$, hence $m_i>mi/r$, for some $i<r$. The composition $h^*$ with partial sums $H^*_j=m_j$ for $j<r$ and $H^*_r=m$ has positive parts, because $m_j-m_{j-1}\geq1$. It belongs to $\Hq$ but does not lie below the diagonal. Its class also contains the composition strictly below the diagonal, which belongs to $\Hq$ as well. Hence this class meets $\Hq$ in at least two elements, and $|\Hq|>|\Cq|/r$.
		\end{itemize}
		By Lemma~\ref{lem:record-floor}, equality holds in the lower bound if and only if $r$ is a lower record index.
		
		\emph{The upper bound.} Let $\Dq$ be the set of compositions of $t$ into $r$ parts, so $|\Dq|=\binom{m}{r-1}$ and $\Wq\subseteq\Dq$. Every $b\in\Wq$ satisfies $B_i\leq m_i<i\alpha<it/r$ for $i<r$, so it lies strictly below the diagonal. By Lemma~\ref{lem:rotation}(b), $\Wq$ meets only classes with $r$ elements, and each of them in at most one element. There are at most $|\Dq|/r$ such classes, so $a_q(r)\leq|\Dq|/r$, with equality if and only if every class has $r$ elements and meets $\Wq$.
		
		If $\gcd(r,t)>1$, some class has fewer than $r$ elements, and the inequality is strict. Let $\gcd(r,t)=1$. As before, a composition lies strictly below the diagonal if and only if $B_i\leq\lfloor ti/r\rfloor$ for $i<r$, and every class contains exactly one such composition. Since $m_i<it/r$, we have $m_i\leq\lfloor ti/r\rfloor$.
		\begin{itemize}
			\item If $m_i=\lfloor ti/r\rfloor$ for all $i<r$, then $\Wq$ is precisely the set of compositions strictly below the diagonal, and $|\Wq|=|\Dq|/r$.
			\item Otherwise $m_i<\lfloor ti/r\rfloor$ for some $i<r$. The composition with partial sums $\lfloor tj/r\rfloor$, $1\leq j\leq r$, has positive parts because $t/r>1$, and lies strictly below the diagonal, but it does not belong to $\Wq$. Its class does not meet $\Wq$, and $|\Wq|<|\Dq|/r$.
		\end{itemize}
		By Lemma~\ref{lem:record-floor}, equality holds in the upper bound if and only if $r$ is an upper record index.
		
		Finally, for $r\geq2$ the two bounds differ, since their quotient is $m_r/(m_r-r+1)>1$. Hence $\Lq_q\cap\Uq_q=\{1\}$.
	\end{proof}
	
	\begin{remark} \label{rem:earlier-bounds}
		An alternative proof of the lower bound in Theorem~\ref{thm:extremal-orders} uses a weighted lattice-path identity of Nakamigawa and Tokushige~\cite[Theorem~1]{NakamigawaTokushige12}. Let $r\geq2$, $m=m_r$, $\delta_r=r\alpha-m\in(0,1)$, and $\kappa=\delta_r/(\alpha-1)$. Deleting the final $0$ from a $q$-admissible word gives a path from $(m-r,r)$ to the origin, with steps $(0,-1)$ for $1$ and $(-1,0)$ for $0$. The line $\Lambda$ given by $y=x/(\alpha-1)+\kappa$ passes through $(m-r,r)$. The prefix inequalities identify admissible words with paths staying strictly below $\Lambda$ after their initial point. Let $\mathcal V$ be the set of all paths with these endpoints and steps. For $w\in\mathcal V$, let $\Delta(w)$ be the minimum vertical distance from $\Lambda$ at vertices after the initial point if $w$ stays strictly below $\Lambda$, and set $\Delta(w)=0$ otherwise. The weighted identity, specialized to this configuration, yields
		\[
		\frac{\kappa}{m}\binom{m}{r}=\sum_{w\in\mathcal V}\Delta(w)\leq \min\{1,\kappa\}\,a_q(r).
		\]
		Indeed, every admissible path begins with a vertical step, so $\Delta(w)\leq1$; its endpoint also gives $\Delta(w)\leq\kappa$. Hence $a_q(r)\geq\binom{m}{r}/m$, since $\kappa/\min\{1,\kappa\}\geq1$.
		
		For $q=3$, Hikawa~\cite{HikawaJuly26} writes $W(k,d)$ for the number of binary words $v$ of length $k$ and weight $d$ satisfying $j<(\log_2 3)(v_1+\cdots+v_j)$ for every $1\leq j\leq k$. Deleting the final $0$ of a $3$-admissible word gives a word counted by $W(m_r,r)$; conversely, appending $0$ to such a word gives a $3$-admissible word, since $3^r<2^{m_r+1}$. Hence $a_3(r)=W(m_r,r)$. Hikawa~\cite[Proposition~6.2]{HikawaJuly26} obtained the same lower bound $\binom{m_r}{r}/m_r$. Its upper bound $\binom{m_r-1}{r-1}$ is weaker than the upper bound of Theorem~\ref{thm:extremal-orders} by a factor of $r(m_r-r+1)/m_r$.
	\end{remark}
	
	\begin{remark}
		In terms of the map $T_q$, Theorem~\ref{thm:extremal-orders} says the following. Let $n$ have coefficient stopping time $m_r+1$. Then $q^r/2^{m_r}$ is the coefficient $q^{j}/2^{k}$ in \eqref{eq:2-1} at $k=m_r$, that is, at the last step before the coefficient drops below $1$. The lower bound is attained exactly when this last coefficient is smaller than the corresponding coefficients $q^j/2^{m_j}$ for all orders $j<r$.
	\end{remark}

	\section{Continued fractions and semiconvergents} \label{sec:cf}
	
	The record indices of Section~\ref{sec:extremal} admit a continued-fraction characterization. For the basic theory, see Khinchin~\cite{Khinchin64}. Write $\alpha=\log_2q$ and
	\[
	\alpha=[a_0;a_1,a_2,\ldots]
	\]
	with integers $a_0\geq1$ and $a_k\geq1$, and define $P_k$ and $Q_k$ by $P_{-1}=1$, $P_0=a_0$, $Q_{-1}=0$, $Q_0=1$, and
	\[
	P_k=a_kP_{k-1}+P_{k-2},\qquad Q_k=a_kQ_{k-1}+Q_{k-2}\qquad(k\geq1).
	\]
	The fractions $P_k/Q_k$ are the \emph{convergents} of $\alpha$. For $k\geq0$ and $0\leq s\leq a_{k+1}$, excluding $(k,s)=(0,0)$, the fractions
	\[
	\frac{sP_k+P_{k-1}}{sQ_k+Q_{k-1}}
	\]
	are called \emph{semiconvergents} (or intermediate fractions); for $s=0$ and $s=a_{k+1}$ they are the convergents $P_{k-1}/Q_{k-1}$ and $P_{k+1}/Q_{k+1}$. For odd $k$ these fractions lie below $\alpha$, and for even $k$ they lie above $\alpha$.
	
	A fraction $p/r<\alpha$ with $r\geq1$ is a \emph{best lower approximation of the second kind} if $r\alpha-p<j\alpha-p'$ for all integers $1\leq j<r$ and $p'$ with $p'/j<\alpha$; equivalently, $p=m_r$ and $\delta_r<\delta_j$ for $1\leq j<r$. Similarly, $p/r>\alpha$ is a \emph{best upper approximation of the second kind} if $p=m_r+1$ and $1-\delta_r<1-\delta_j$ for $1\leq j<r$. Thus, for $r\geq2$, the lower and upper record indices of Section~\ref{sec:extremal} are precisely the denominators of best lower and best upper approximations of the second kind. Kimberling \cite{Kimberling97} described best lower and upper approximations in terms of convergents and intermediate fractions. Han\v{c}l and Turek \cite[Theorem~4.5]{HanclTurek} showed that one-sided best approximations of the first and of the second kind coincide and gave the parametrization used in the following theorem.
	
	\begin{theorem} \label{thm:cf}
		We have
		\begin{equation}
			\Lq_q= \left\{ sQ_k+Q_{k-1}: k\geq1\ \text{odd},\ 0\leq s<a_{k+1} \right\} \label{eq:5-1}
		\end{equation}
		and
		\begin{equation}
			\Uq_q= \left\{ sQ_k+Q_{k-1}: k\geq0\ \text{even},\ 0\leq s<a_{k+1},\ (k,s)\neq(0,0) \right\}. \label{eq:5-2}
		\end{equation}
		Thus $\Lq_q$ is the set of denominators of the convergents and semiconvergents below $\alpha$, $\Uq_q$ is the set of denominators of those above $\alpha$, and $\Lq_q\cap\Uq_q=\{1\}$.
	\end{theorem}
	
	\begin{proof}
		By \cite[Theorem~4.5]{HanclTurek}, translated to our indexing, the best lower and upper approximations of the second kind are the fractions
		\[
		\frac{sP_k+P_{k-1}}{sQ_k+Q_{k-1}}, \qquad 0\leq s<a_{k+1},
		\]
		with odd $k$ for approximations from below and even $k$ for approximations from above; the pair $(k,s)=(0,0)$ is excluded because $Q_{-1}=0$. By Theorem~\ref{thm:extremal-orders} and the preceding discussion, the denominators $r\geq2$ of these fractions are the elements $r\geq2$ of $\Lq_q$ and $\Uq_q$, respectively. It remains to check that $1$ occurs on both sides, as it must by Theorem~\ref{thm:extremal-orders}. On the lower side, $1=Q_0$ arises from $(k,s)=(1,0)$. On the upper side, $1$ arises from $(k,s)=(0,1)$ if $a_1\geq2$, and from $(k,s)=(2,0)$, with $Q_1=a_1=1$, if $a_1=1$. Since $s$ ranges over $0\leq s<a_{k+1}$, every convergent and semiconvergent other than $P_{-1}/Q_{-1}$ occurs, which gives the second statement.
	\end{proof}
	
	Both sets are infinite, since the convergents $P_k/Q_k$ lie alternately below and above $\alpha$. For $q=3$,
	\[
	\log_2 3=[1;1,1,2,2,3,1,5,2,23,\ldots],
	\]
	and
	\[
	\Lq_3=\{1,2,7,12,53,359,665,\ldots\}, \qquad \Uq_3=\{1,3,5,17,29,41,94,147,200,253,306,971,\ldots\}.
	\]
	The union $\Lq_3\cup\Uq_3$ is the set of terms of \seqnum{A206788}, the denominators of semiconvergents to $\log_2 3$; that entry lists the denominator $1$ twice, once for each side. The corresponding sequences for $q=5$ and $q=7$ are \seqnum{A400205} and \seqnum{A400665}. Table~\ref{tab:records} lists the elements up to $20$ of $\Lq_q$ and $\Uq_q$ for $q=3,5,7$. They agree with the equality cases of \eqref{eq:4-4} obtained by computing $a_q(r)$ directly.
	
	\begin{table}[ht]
		\centering
		\small
		\begin{tabular}{c c p{5.0cm} p{3.6cm}}
			\toprule
			$q$ & continued fraction of $\log_2q$ & $\Lq_q\cap[1,20]$ & $\Uq_q\cap[1,20]$\\
			\midrule
			3 & $[1;1,1,2,2,3,1,5,\ldots]$ & $1,2,7,12$ & $1,3,5,17$\\
			5 & $[2;3,9,2,2,4,6,\ldots]$ & $1,4,7,10,13,16,19$ & $1,2,3$\\
			7 & $[2;1,4,5,4,5,4,\ldots]$ & $1,2,3,4,5$ & $1,6,11,16$\\
			\bottomrule
		\end{tabular}
		\caption{Orders $r\leq20$ for which equality holds in the lower and upper bounds of Theorem~\ref{thm:extremal-orders}.} \label{tab:records}
	\end{table}
	
	For $q=3$, several related sequences are in the OEIS: $m_r$ is \seqnum{A056576}, the length $m_r+1$ of the admissible sequences of order $r$ is \seqnum{A020914}, the continued fraction of $\log_2 3$ is \seqnum{A028507}, and the denominators $Q_k$ of its convergents are \seqnum{A005664}.

	\section{Rational Catalan numbers} \label{sec:catalan}
	
	A \emph{lattice path} from $(0,0)$ to $(a,b)$ is a sequence of $a$ steps $E=(1,0)$ and $b$ steps $N=(0,1)$. For coprime positive integers $a$ and $b$, let $\Cat(a,b)$ be the number of lattice paths from $(0,0)$ to $(a,b)$ all of whose points $(x,y)$ satisfy $ay\leq bx$, that is, which never go above the line from $(0,0)$ to $(a,b)$. Bizley \cite{Bizley54} proved
	\begin{equation}
		\Cat(a,b)=\frac1{a+b}\binom{a+b}{a}. \label{eq:6-1}
	\end{equation}
	These numbers are called \emph{rational Catalan numbers}, and the paths are called rational Dyck paths; see Armstrong, Rhoades, and Williams \cite{ArmstrongRW13}. For $b=a+1$, $\Cat(a,a+1)$ is the ordinary Catalan number $\frac1{a+1}\binom{2a}{a}$ (\seqnum{A000108}). The OEIS entry \seqnum{A100982} already contains lattice path interpretations of $a_3(r)$, including a link by van Tol and cross-references to \seqnum{A060941} and \seqnum{A293946}.
	
	\begin{theorem} \label{thm:dyck}
		Let $r\geq2$. If $r\in\Lq_q$, then $\gcd(r,m_r-r)=1$, the set $\Aq_q(r)$ is in bijection with the lattice paths from $(0,0)$ to $(r,m_r-r)$ that never go above the line between these points, and
		\begin{equation}
			a_q(r)=\Cat(r,m_r-r) =\frac1{m_r}\binom{m_r}{r}. \label{eq:6-2}
		\end{equation}
		If $r\in\Uq_q$, then $\gcd(r,m_r+1-r)=1$, the same statement holds with $(r,m_r+1-r)$ in place of $(r,m_r-r)$, and
		\begin{equation}
			a_q(r)=\Cat(r,m_r+1-r) =\frac1{m_r+1}\binom{m_r+1}{r}. \label{eq:6-3}
		\end{equation}
	\end{theorem}
	
	\begin{proof}
		We first note how a sequence of nonnegative integers $y=(y_1,\ldots,y_r)$ with sum $s>0$ and partial sums $Y_i$ defines the lattice path $E\,N^{y_1}\,E\,N^{y_2}\cdots E\,N^{y_r}$ from $(0,0)$ to $(r,s)$. Its highest point with abscissa $i$ is $(i,Y_i)$, so the path never goes above the line from $(0,0)$ to $(r,s)$ if and only if $Y_i\leq si/r$, that is, $Y_i\leq\lfloor si/r\rfloor$, for $1\leq i<r$. Every path with this property begins with $E$, because the point $(0,1)$ lies above the line. Hence this construction is a bijection between such sequences $y$ satisfying $Y_i\leq\lfloor si/r\rfloor$ for $i<r$ and the paths counted by $\Cat(r,s)$.
		
		Let $r\in\Lq_q$, $r\geq2$, and $m=m_r$. By Lemma~\ref{lem:record-floor}, $\gcd(r,m)=1$ and $m_i=\lfloor mi/r\rfloor$ for $i<r$. In the proof of Theorem~\ref{thm:extremal-orders} we saw that $\Aq_q(r)$ is in bijection with the set $\Hq$ of compositions $h$ of $m$ into $r$ parts with $H_i\leq m_i$ for $i<r$. Put $y_i=h_i-1$ and $s=m-r$, which is positive by \eqref{eq:4-3}. Then $y_i\geq0$, $Y_r=s$, and
		\[
		Y_i=H_i-i\leq\left\lfloor\frac{mi}{r}\right\rfloor-i=\left\lfloor\frac{si}{r}\right\rfloor\qquad(1\leq i<r),
		\]
		and every $y$ with these properties arises from exactly one $h\in\Hq$. Since $\gcd(r,s)=\gcd(r,m)=1$, the first paragraph and \eqref{eq:6-1} give \eqref{eq:6-2}.
		
		Let $r\in\Uq_q$, $r\geq2$, and $t=m_r+1$. By Lemma~\ref{lem:record-floor}, $\gcd(r,t)=1$ and $m_i=\lfloor ti/r\rfloor$ for $i<r$. Now $\Aq_q(r)$ is in bijection with the set $\Wq$ of admissible compositions $b$ of $t$. Put $y_i=b_i-1$ and $s=t-r$. Then
		\[
		Y_i=B_i-i\leq\left\lfloor\frac{ti}{r}\right\rfloor-i=\left\lfloor\frac{si}{r}\right\rfloor\qquad(1\leq i<r),
		\]
		and $\gcd(r,s)=\gcd(r,t)=1$. As before, this gives \eqref{eq:6-3}.
	\end{proof}
	
	For $q=3$, Theorem~\ref{thm:dyck} gives
	\[
	a_3(7)=\Cat(7,4)=30,\qquad a_3(12)=\Cat(12,7)=2652,\qquad a_3(17)=\Cat(17,10)=312455,
	\]
	where $7,12\in\Lq_3$ and $17\in\Uq_3$.
	
	For $q=5$ and $r=2,3,4$, Theorem~\ref{thm:dyck} gives $a_5(r)=\Cat(r,r+1)$, the ordinary Catalan numbers $2,5,14$. This does not continue: $a_5(5)=56$, whereas the fifth Catalan number is $42$.
	
	For $q=7$, the first six terms of \seqnum{A174796} are
	\[
	a_7(1),\ldots,a_7(6)=1,2,7,30,143,728,
	\]
	which are the terms $\seqnum{A006013}(0),\ldots,\seqnum{A006013}(5)$ of the sequence $\frac1{k+1}\binom{3k+1}{k}$. Indeed, for $2\leq r\leq5$ we have $r\in\Lq_7$ and $m_r=3r-1$, while $6\in\Uq_7$ and $m_6+1=17=3\cdot6-1$. In all five cases Theorem~\ref{thm:dyck} gives
	\[
	a_7(r)=\Cat(r,2r-1) =\frac1r\binom{3r-2}{r-1} =\seqnum{A006013}(r-1).
	\]
	The agreement ends at $r=7\notin\Lq_7\cup\Uq_7$: $a_7(7)=3148$, whereas $\seqnum{A006013}(6)=3876$.

	\section{Dependence on \texorpdfstring{$q$}{q}} \label{sec:comparison}
	
	By \eqref{eq:3-2}, $a_q(r)=c_r(\log_2q)$, where $c_r(\alpha)$ counts the sequences $1\leq s_1<\cdots<s_{r-1}$ with $s_i\leq\lfloor i\alpha\rfloor$. The following proposition compares these numbers for different values of $\alpha$.
	
	\begin{proposition} \label{prop:monotonicity}
		Let $1<\alpha<\beta$ and $r\geq2$. Then $c_r(\alpha)\leq c_r(\beta)$, and
		\begin{equation}
			c_r(\alpha)=c_r(\beta) \iff \lfloor i\alpha\rfloor=\lfloor i\beta\rfloor \quad(1\leq i<r). \label{eq:7-1}
		\end{equation}
		Thus $\alpha\mapsto c_r(\alpha)$ is nondecreasing, and it is constant on a set of values of $\alpha$ if and only if the integers $\lfloor\alpha\rfloor,\lfloor2\alpha\rfloor,\ldots,\lfloor(r-1)\alpha\rfloor$ are the same for all values in this set.
	\end{proposition}
	
	\begin{proof}
		Put $u_i=\lfloor i\alpha\rfloor$ and $v_i=\lfloor i\beta\rfloor$. Since $u_i\leq v_i$, every sequence counted by $c_r(\alpha)$ is counted by $c_r(\beta)$, and $c_r(\alpha)\leq c_r(\beta)$. If $u_i=v_i$ for all $i<r$, then $c_r(\alpha)=c_r(\beta)$.
		
		Conversely, assume $u_j<v_j$ for some $j<r$. Define
		\[
		s_i=i\quad(1\leq i<j),\qquad s_i=u_j+1+i-j\quad(j\leq i<r).
		\]
		Since $\alpha>1$, we have $u_j\geq j$, so this is a strictly increasing sequence of positive integers, and it satisfies $s_j=u_j+1>u_j$. Hence it is not counted by $c_r(\alpha)$. On the other hand, $s_i=i\leq u_i\leq v_i$ for $i<j$, $s_j=u_j+1\leq v_j$, and, for $j<i<r$, $s_i=s_j+i-j\leq v_j+(v_i-v_j)=v_i$, because $v_i-v_j\geq i-j$ for $\beta>1$. Hence the sequence is counted by $c_r(\beta)$, and $c_r(\alpha)<c_r(\beta)$.
	\end{proof}
	
	\begin{corollary} \label{cor:three-sequences}
		For every $r\geq2$ we have $a_3(r)<a_5(r)$. Moreover $a_5(2)=a_7(2)=2$, and $a_5(r)<a_7(r)$ for $r\geq3$.
	\end{corollary}
	
	\begin{proof}
		We have $\log_23<\log_25<\log_27$. Since $\lfloor\log_23\rfloor=1<2=\lfloor\log_25\rfloor$, Proposition~\ref{prop:monotonicity} gives $a_3(r)<a_5(r)$ for $r\geq2$. Next, $\lfloor\log_25\rfloor=\lfloor\log_27\rfloor=2$, which gives $a_5(2)=a_7(2)$, while $\lfloor2\log_25\rfloor=4<5=\lfloor2\log_27\rfloor$, which gives $a_5(r)<a_7(r)$ for $r\geq3$.
	\end{proof}

	\section{Growth} \label{sec:growth}
	
	For $\alpha>1$ put
	\begin{equation}
		B(\alpha)= \frac{\alpha^\alpha}{(\alpha-1)^{\alpha-1}}, \qquad B_q=B(\log_2q). \label{eq:8-1}
	\end{equation}
	
	For $q=3$, Lagarias~\cite[Theorem~D]{Lagarias85} established the cumulative upper bound $1-F(k)\leq2^{-\eta k}$, where $F(k)$ denotes the natural density of positive integers whose ordinary stopping time is at most $k$, and $\eta=1-h(1/\log_2 3)$ with $h$ the binary entropy function. This cumulative estimate differs from the asymptotic rates for the individual counts $a_q(r)$ below. The exponential rate and the associated density decay were obtained by Hikawa~\cite[Theorem~6.3]{HikawaJuly26}; see also the later presentation~\cite[Theorem~7.3]{HikawaSeptember26}. The corresponding rates hold for every odd $q\geq3$.
	
	\begin{corollary} \label{cor:growth}
		For every odd $q\geq3$,
		\begin{equation}
			\lim_{r\to\infty}a_q(r)^{1/r}=B_q \qquad\text{and}\qquad \lim_{r\to\infty}\left(\frac{a_q(r)}{2^{m_r+1}}\right)^{1/r}=\frac{B_q}{q}<1. \label{eq:8-2}
		\end{equation}
		The sequence $(B_q)$ is strictly increasing in $q$. In particular,
		\begin{equation}
			\lim_{r\to\infty} \left(\frac{a_5(r)}{a_3(r)}\right)^{1/r} =\frac{B_5}{B_3} \approx1.7218785536, \qquad \lim_{r\to\infty} \left(\frac{a_7(r)}{a_5(r)}\right)^{1/r} =\frac{B_7}{B_5} \approx1.2726660717. \label{eq:8-3}
		\end{equation}
	\end{corollary}
	
	\begin{proof}
		Put $\alpha=\log_2q$. By Theorem~\ref{thm:extremal-orders}, $a_q(r)$ lies between $\binom{m_r-1}{r-1}/r$ and $\binom{m_r}{r-1}/r$. Since $m_r=\alpha r+O(1)$, Stirling's formula gives
		\[
		\frac1r\log\binom{m_r-1}{r-1} = \alpha\log\alpha -(\alpha-1)\log(\alpha-1)+o(1),
		\]
		and the same holds with $m_r$ in place of $m_r-1$. Hence both bounds have $r$th roots converging to $B(\alpha)$, which proves the first limit. Since $2^{m_r+1}=2^{1-\delta_r}q^r$ with $0<\delta_r<1$, the second limit equals $B_q/q$. To see that $B_q<q$, let $h(p)=-p\log_2p-(1-p)\log_2(1-p)$ be the binary entropy function. A direct computation gives $\log_2B(\alpha)=\alpha\,h(1/\alpha)$. Since $h(p)<1$ for $p\neq1/2$ and $\alpha\neq2$, we obtain $\log_2B_q<\alpha=\log_2q$. Finally,
		\[
		\frac{d}{d\alpha}\log B(\alpha) =\log\frac{\alpha}{\alpha-1}>0,
		\]
		so $B_q$ is strictly increasing in $q$, and \eqref{eq:8-3} follows from the first limit.
	\end{proof}
	
	By Proposition~\ref{prop:density}, the second limit in \eqref{eq:8-2} is the exponential rate at which the density of the integers with coefficient stopping time $m_r+1$ decreases.
	
	\begin{corollary} \label{cor:theta}
		For each fixed odd $q\geq3$, as $r\to\infty$, we have
		\[
		a_q(r)=\Theta_q\!\left(\frac{B_q^r}{r^{3/2}}\right), \qquad
		\frac{a_q(r)}{2^{m_r+1}}=\Theta_q\!\left(\frac{(B_q/q)^r}{r^{3/2}}\right).
		\]
		Here $\Theta_q$ indicates two-sided estimates with positive implicit constants depending only on $q$. For $\lambda=\log_2 3$, put $W(d)=\sum_{d\leq k<\lambda d}W(k,d)$, with $W(k,d)$ as in Remark~\ref{rem:earlier-bounds} and~\cite{HikawaJuly26}. Hikawa's constant $\gamma=\lambda h(1/\lambda)$ satisfies $2^\gamma=B_3$ by the identity $\log_2 B(\alpha)=\alpha h(1/\alpha)$ in the proof of Corollary~\ref{cor:growth}. In particular,
		\[
		W(d)=\Theta\!\left(\frac{2^{\gamma d}}{d^{3/2}}\right).
		\]
		With $\rho_d=W(d)/2^{\lceil\lambda d\rceil}$ and $c=\lambda-\gamma$, the same estimate gives $\rho_d=\Theta(d^{-3/2}2^{-cd})$. Thus both formulations of Conjecture~7.1 in \cite{HikawaJuly26}, and Conjecture~9.1 in \cite{HikawaSeptember26}, hold.
	\end{corollary}
	
	\begin{proof}
		Put $\alpha=\log_2q$ and $\delta_r=\{r\alpha\}$. By Theorem~\ref{thm:extremal-orders} and the binomial identities,
		\[
		\frac{1}{m_r}\binom{m_r}{r}\leq a_q(r)\leq
		\frac{1}{m_r-r+1}\binom{m_r}{r}.
		\]
		Stirling's formula, applied to $m_r$, $r$, and $m_r-r$, with $m_r=\alpha r-\delta_r$ and $0<\delta_r<1$, yields
		\[
		\binom{m_r}{r}=
		\sqrt{\frac{\alpha}{2\pi(\alpha-1)}}
		\left(\frac{\alpha-1}{\alpha}\right)^{\delta_r}
		\frac{B_q^r}{\sqrt r}\bigl(1+O_q(r^{-1})\bigr).
		\]
		The error term is uniform in $\delta_r$: expanding $\log m_r$ and $\log(m_r-r)$ gives $\log m_r=\log(\alpha r)-\delta_r/(\alpha r)+O_q(r^{-2})$ and $\log(m_r-r)=\log((\alpha-1)r)-\delta_r/((\alpha-1)r)+O_q(r^{-2})$, and their contribution to the logarithm of the binomial coefficient is $\delta_r\log((\alpha-1)/\alpha)+O_q(r^{-1})$. The factor depending on $\delta_r$ lies between $(\alpha-1)/\alpha$ and $1$. Moreover, $m_r<\alpha r$ and $m_r-r+1>(\alpha-1)r$. Thus the displayed binomial bounds prove the first estimate; the second follows from $2^{m_r+1}=2^{1-\delta_r}q^r$. Finally, $W(d)=a_3(d+1)$ for $d\geq1$. As in the proof of Lemma~\ref{lem:structure}, a word $v$ counted by $W(d)$, with length $k$ and ones at $p_1<\cdots<p_d$, is characterized by $p_1=1$, $p_{i+1}-1\leq m_i$ for $i<d$, and $k\leq m_d$. Hence $v\mapsto(p_2-1,\ldots,p_d-1,k)$ is a bijection onto the sequences $1\leq s_1<\cdots<s_d$ with $s_i\leq m_i$, which are counted by $a_3(d+1)$ in Lemma~\ref{lem:structure}. In the form $X(d+1)=W(d)$, where $X(d)$ counts the just-converged parity vectors of weight $d$, the identity appears in Hikawa and Nakanishi~\cite[(2-7), Theorem~2]{HikawaNakanishi26}; see also Hikawa's bijection~\cite[Theorem~3.3]{HikawaJuly26}. The claimed estimate follows by shifting the index.
	\end{proof}
	
	Hikawa called the conjectured factor $d^{-3/2}$ a ballot-type refinement~\cite[Conjecture~7.1]{HikawaJuly26}. Here this factor follows directly from the explicit binomial bounds of Theorem~\ref{thm:extremal-orders}.
	
	The quotient of $a_q(r)$ and the lower bound in \eqref{eq:4-4} oscillates as follows. Define
	\begin{equation}
		R_q(r)= \frac{r\,a_q(r)}{\binom{m_r-1}{r-1}}. \label{eq:8-4}
	\end{equation}
	
	\begin{corollary} \label{cor:oscillation}
		Let $q\geq3$ be odd and $\alpha=\log_2q$. Then $1\leq R_q(r)<\alpha/(\alpha-1)$ for all $r\geq1$, and
		\begin{equation}
			\liminf_{r\to\infty}R_q(r)=1, \qquad \limsup_{r\to\infty}R_q(r) =\frac{\alpha}{\alpha-1}. \label{eq:8-5}
		\end{equation}
		More precisely, $R_q(r)=1$ if and only if $r\in\Lq_q$, and $R_q(r)=m_r/(m_r-r+1)$ for $r\in\Uq_q$.
	\end{corollary}
	
	\begin{proof}
		Theorem~\ref{thm:extremal-orders} gives $1\leq R_q(r)\leq m_r/(m_r-r+1)$, with equality on the left exactly for $r\in\Lq_q$ and on the right exactly for $r\in\Uq_q$. Both sets are infinite by Theorem~\ref{thm:cf}. Since $m_r/r\to\alpha$, the values $m_r/(m_r-r+1)$ converge to $\alpha/(\alpha-1)$, which proves \eqref{eq:8-5}. Finally,
		\[
		\frac{m_r}{m_r-r+1}<\frac{\alpha}{\alpha-1}
		\iff m_r>(r-1)\alpha \iff \delta_r<\alpha,
		\]
		which holds because $\delta_r<1<\alpha$.
	\end{proof}
	
	Table~\ref{tab:threeq} lists the constants to ten decimal places.
	
	\begin{table}[ht]
		\centering
		\small
		\renewcommand{\arraystretch}{1.15}
		\begin{tabular}{c c c c c}
			\toprule
			$q$ & $\log_2q$ & $B_q$ & $B_q/q$ & $\log_2q/(\log_2q-1)$\\
			\midrule
			3 & 1.5849625007 & 2.8395137305 & 0.9465045768 & 2.7095112914\\
			5 & 2.3219280949 & 4.8892977952 & 0.9778595590 & 1.7564707974\\
			7 & 2.8073549221 & 6.2224434185 & 0.8889204884 & 1.5532947557\\
			\bottomrule
		\end{tabular}
		\caption{The constants of Corollaries~\ref{cor:growth} and~\ref{cor:oscillation}.} \label{tab:threeq}
	\end{table}

	\section*{Acknowledgments}
	The author thanks Ruud H.~G. van Tol for providing a Lean~4 formalization of the model behind \seqnum{A100982}, including the case $q=3$ of Proposition~\ref{prop:first-violation-range} and a machine-checked derivation of Corollary~\ref{cor:A100982-recurrence}. The author also thanks the editors of the \emph{Journal of Integer Sequences} for comments on an earlier version, particularly for noting that Proposition~\ref{prop:first-violation-range} holds for every nondecreasing bound.

	\section*{Use of large language models}
	ChatGPT (OpenAI) and Claude (Anthropic) were used to examine mathematical arguments, research relevant literature, develop and verify calculations, and revise the text. The language models assisted in formulating and refining arguments concerning recurrence, binomial bounds, cases of equality, and asymptotic results. The proposed arguments were verified against definitions and through precise calculations. Additional checks were carried out in further sessions with the AI models, independent of the original creation phase; these identified gaps in the proofs as well as inaccuracies in citations, which were subsequently corrected. The author bears responsibility for the final statements, proofs, calculations, and citations.

	\smallskip
	\hrule
	\smallskip
	
	\noindent 2020 {\it Mathematics Subject Classification}: Primary 11B83; Secondary 11A55, 05A15, 11B37.
	
	\noindent \emph{Keywords:} $3x+1$ problem; $qx+1$ map; coefficient stopping time; parity sequence; admissible sequence; cycle lemma; continued fraction; semiconvergent; rational Catalan number.
	
	\smallskip
	\hrule
	\bigskip
	
	\noindent (Concerned with sequences \seqnum{A000108}, \seqnum{A005664}, \seqnum{A006013}, \seqnum{A020914}, \seqnum{A028507}, \seqnum{A056576}, \seqnum{A060941}, \seqnum{A100982}, \seqnum{A174795}, \seqnum{A174796}, \seqnum{A206788}, \seqnum{A293946}, \seqnum{A400205}, and \seqnum{A400665}.)
	
	\bigskip
	\hrule
	
\end{document}